\documentclass[11pt,amstex,amssymb]{amsart}
\usepackage{amsmath,amsthm,amsfonts,amssymb,amscd}
\usepackage[latin1]{inputenc}
\usepackage[all]{xy}
\usepackage[dvips]{graphicx}
\usepackage{amsmath}
\usepackage{amsthm}
\usepackage{amsfonts}
\usepackage{xcolor}

\usepackage{amsmath}
\usepackage{amssymb}
\usepackage{amsthm}

\def\fa{{\mathcal{F}}}
\def\O{{\mathcal{O}}}

\def\M{{\mathcal{M}}}
\def\D{{\mathcal{D}}}
\def\U{{\mathcal{U}}}
\def\B{{\mathcal{B}}}

  \newcommand{\dsum}{\displaystyle\sum}

\title{Irreducible holonomy groups and Riccati foliations in higher complex dimension}
\author{V. Le\'on}
\author{M. Martelo}
\author{B. Sc\'ardua}
\address{V. Le\'on. ILACVN - CICN, Universidade Federal da Integração Latino-Americano, Parque tecnológico de Itaipu, Foz do Iguaçu-PR, 85867-970 - Brazil}
\email{victor.leon@unila.edu.br}

\address{M. Martelo. Instituto de Matem\'atica - Universidade Federal Fluminense, Niter\'oi -
Rio de Janeiro-RJ, 24210-201 - Brazil}
\email{mitchaelmartelo@id.uff.br}

\address{B. Sc\'ardua. Instituto de Matem\'atica - Universidade Federal do Rio de Janeiro,
CP. 68530-Rio de Janeiro-RJ, 21945-970 - Brazil}
\email{scardua@im.ufrj.br}


\date{}

\begin{document}

\begin{abstract}
We study  groups of germs of complex diffeomorphisms having a property called
{\it irreducibility}. The notion is motivated by a similar property of the fundamental group of the complement of an irreducible hypersurface in the complex projective space. Natural examples of such groups of germ maps are given by holonomy groups and monodromy groups of integrable systems (foliations) under certain conditions on the singular or ramification set. The case of complex dimension one is studied in  \cite{VMB} where finiteness is proved for irreducible groups under certain arithmetic hypothesis on the linear part. In dimension $n \geq 2$ the picture changes since linear groups are not always abelian in dimension two or bigger. Nevertheless,  we still obtain a finiteness result under some   conditions in the linear part of the group, for instance if the linear part is abelian. Examples are given illustrating the role of our hypotheses.
Applications are given to the framework of holomorphic foliations and analytic deformations of rational fibrations by Riccati foliations.
 \end{abstract}

\maketitle

\newtheorem{Theorem}{Theorem}[section]
\newtheorem{Corollary}{Corollary}[section]
\newtheorem{Proposition}{Proposition}[section]
\newtheorem{Lemma}{Lemma}[section]
\newtheorem{Claim}{Claim}[section]
\newtheorem{Definition}{Definition}[section]
\newtheorem{Example}{Example}[section]
\newtheorem{Remark}{Remark}[section]
\newtheorem*{cltheorem}{Theorem}
\newtheorem{Question}{Question}[section]

\newcommand\virt{\rm{virt}}
\newcommand\SO{\rm{SO}}
\newcommand\G{\varGamma}
\newcommand\Om{\Omega}
\newcommand\Kbar{{K\kern-1.7ex\raise1.15ex\hbox to 1.4ex{\hrulefill}}}
\newcommand\codim{\rm{codim}}
\renewcommand\:{\colon}
\newcommand\s{\sigma}
\def\vol#1{{|{\bfS}^{#1}|}}

\def\fa{{\mathcal F}}
\def\H{{\mathcal H}}
\def\O{{\mathcal O}}
\def\P{{\mathcal P}}
\def\L{{\mathcal L}}
\def\C{{\mathcal C}}
\def\Z{{\mathcal Z}}

\def\M{{\mathcal M}}
\def\N{{\mathcal N}}
\def\R{{\mathcal R}}
\def\ea{{\mathcal e}}
\def\Oa{{\mathcal O}}
\def\ee{{\bfE}}

\def\A{{\mathcal A}}
\def\B{{\mathcal B}}
\def\H{{\mathcal H}}
\def\V{{\mathcal V}}
\def\U{{\mathcal U}}
\def\al{{\alpha}}
\def\be{{\beta}}
\def\ga{{\gamma}}
\def\Ga{{\Gamma}}
\def\om{{\omega}}
\def\Om{{\Omega}}
\def\La{{\Lambda}}
\def\ov{\overline}
\def\dd{{\bfD}}
\def\pp{{\bfP}}

\def\nn{{\mathbb N}}
\def\zz{{\mathbb Z}}
\def\bq{{\mathbb Q}}
\def\bp{{\mathbb P}}
\def\bd{{\mathbb D}}
\def\bh{{\mathbb H}}
\def\te{{\theta}}
\def\rr{{\mathbb R}}
\def\bb{{\mathbb B}}
\def\pp{{\mathbb P}}
\def\dd{{\mathbb D}}
\def\zz{{\mathbb Z}}
\def\qq{{\mathbb Q}}
\def\hh{{\mathbb H}}
\def\nn{{\mathbb N}}
\def\LL{{\mathbb L}}
\def\co{{\mathbb C}}
\def\qq{{\mathbb Q}}
\def\na{{\mathbb N}}
\def\esima{${}^{\text{\b a}}$}
\def\esimo{${}^{\text{\b o}}$}
\def\lg{\lambdangle}
\def\rg{\rangle}
\def\ro{{\rho}}
\def\lV{\left\Vert}
\def\rV{\right\Vert }
\def\lv{\left\vert}
\def\rv{\right\vert }
\def\Sa{{\mathcal S}}
\def\D{{\mathcal D  }}

\def\si{{\bf S}}
\def\ve{\varepsilon}
\def\vr{\varphi}
\def\lV{\left\Vert }
\def\rV{\right\Vert}
\def\lv{\left\vert }
\def\rv{\right\vert}
\def\Range{\rm{{R}}}
\def\vol{\rm{{Vol}}}
\def\ind{\rm{{i}}}

\def\Int{\rm{{Int}}}
\def\Dom{\rm{{Dom}}}
\def\supp{\rm{{supp}}}
\def\Aff{\mbox{Aff}}
\def\Exp{\rm{{Exp}}}
\def\Hom{\rm{{Hom}}}
\def\codim{\rm{{codim}}}
\def\cotg{\rm{{cotg}}}
\def\dom{\rm{{dom}}}
\def\Sa{\mathcal{{S}}}

\def\VIP{\rm{{VIP}}}
\def\argmin{\rm{{argmin}}}
\def\Sol{\rm{{Sol}}}
\def\Ker{\rm{{Ker}}}
\def\Sat{\rm{{Sat}}}
\def\diag{\mbox{diag}}
\def\rank{\rm{{rank}}}
\def\Sing{\rm{{Sing}}}
\def\sing{\rm{{sing}}}
\def\hot{\rm{{h.o.t.}}}

\def\Fol{\rm{{Fol}}}
\def\grad{\rm{{grad}}}
\def\id{\rm{{id}}}
\def\Id{\mbox{Id}}
\def\sep{\rm{{Sep}}}
\def\Aut{\mbox{Aut}}
\def\Sep{\rm{{Sep}}}
\def\Res{\rm{{Res}}}
\def\ord{\rm{{ord}}}
\def\h.o.t.{\rm{{h.o.t.}}}
\def\Hol{\mbox{Hol}}
\def\Diff{\mbox{Diff}}
\def\SL{\mbox{SL}}
\def\Aut{\mbox{Aut}}
\def\GL{\mbox{GL}}
\tableofcontents
\section{Introduction}
In  \cite{VMB} we have introduced the notion of irreducible group of germs of diffeomorphisms in dimension $1$.
In that work we give conditions under which such a group is finite and prove some applications of this
to the problem of existence of holomorphic first integrals for codimension one foliations.
In this work we investigate the extension of this to the case of any dimension.
We make the following definition:
\begin{Definition}
\label{Definition:groupgen}{\rm
A group $G$ is {\it irreducible} if it admits a finite set of generators $g_1,\ldots,g_{\nu+1}$ such that:
\begin{itemize}
\item[{\rm(a)}] $g_1\circ \cdots \circ g_{\nu+1}=e_G$
\item[{\rm(b)}] $g_i$ and $g_j$ are conjugate in $G$ for all $i,j$.
\end{itemize}
}
\end{Definition}

We shall refer to $\{g_1,\ldots,g_{\nu+1}\}$ as a {\it basic set of generators}.
 The above definition does not exclude the possibility that $g_i=g_j$.  An irreducible abelian group is finite cyclic: indeed,
since the group is abelian we have $g_i=g_j$, for all $i,j$. Therefore the group is generated by
an element $g_1$ such that  $g_1^{\nu +1} =e_G$.

We shall denote by $\Diff(\mathbb C^n,0)$ the group of germs of complex diffeomorphisms fixing the origin $0 \in \co^n$.
In this work we shall focus on irreducible subgroups $G\subset \Diff(\mathbb C^n,0)$.
It is important to point-out that conditions (a) and (b) in Definition~\ref{Definition:groupgen} are independent and therefore not equivalent (cf. Proposition~\ref{Proposition:examples}).
Every cyclic subgroup {\it of finite order} $G\subset \Diff(\mathbb C^n,0)$ is irreducible.
A first question would be whether finite subgroups of $\Diff(\mathbb C^n,0)$ are also irreducible.
A second, more challenging, is whether irreducible subgroups of $\Diff(\mathbb C^n,0)$ are always finite.
The above questions have negative response even in the linear case (group of matrices) (cf. Proposition~\ref{Proposition:examples}).

From now on we shall consider $\mathbb N=\{0,1,2,\ldots\}$ and $\mathbb N^*=\{1,2,\ldots\}$. Our basic  result is the following:

\begin{Theorem}\label{Theorem:irreducibledimn}{\rm
Let $ G\subset \Diff(\mathbb C^n,0)$ be an irreducible group having a basic set of generators $\{f_1,\ldots,f_{\nu+1}\}$ with the same linear part $Df_j(0)=A\in \GL(n,\mathbb C)$. Assume that the eigenvalues of $A$ are $\lambda_1,\ldots,\lambda_n$ such that $\lambda_j$
is a root of the unit of order $p^{s_j}$, where $p$ is prime and
$s_j\in\mathbb{N}$.
Then $ G$ is finite cyclic. Indeed, $G$ is analytically conjugate to a cyclic group generated by a diagonal matrix of the form $A=\diag(\xi_1,\ldots,\xi_n)$ where $\xi_j$ is a root of the unit of order $p^{r_j}$.}
\end{Theorem}

We point-out that  Theorem~\ref{Theorem:irreducibledimn} holds with an analogous  statement for the case of groups of formal diffeomorphisms. Indeed, the proofs are based on some normal forms for the resonant case and on the Taylor series expansion, so the proofs apply {\it ipsis litteris} to the formal case.\\

The {\it linear part} of a group $G\subset \Diff(\mathbb C^n,0)$ is the subgroup
of $\GL(n,\mathbb C)$ of the linear maps $Df(0)$ where $f\in G$ and the coordinates are the canonical affine coordinates  $(z_1,\ldots,z_n)\in \mathbb C^n$.

\begin{Corollary}\label{Corollary:irreducibledimn}{\rm
Let $ G\subset \Diff(\mathbb C^n,0)$ be an irreducible  group with abelian linear part, having  any (not necessarily basic)  set of  generators $g_1,g_2,\ldots,g_{p^s}\in G$ such that  $g_1\circ \cdots\circ g_{p^s}=\Id$ for some prime number $p$ and some $s \in \mathbb N$.
Then  $ G$ is finite cyclic of order $p^\ell$ for some $\ell \leq s$.}
\end{Corollary}

As for an application to the framework of foliations:

\begin{Theorem}
\label{Theorem:leafholonomy}{\rm
Let $\fa$ be a codimension $n$ holomorphic foliation with singularities on a complex manifold $M^{n+2}$. Assume that there is a leaf $L_0\in \fa$ which is homeomorphic to
 $\mathbb P^2 \setminus C$ where $C \subset \mathbb P^{2}$ is an irreducible algebraic curve of degree $p^s$ for some prime number $p$ and some $s\in \mathbb N$.
Assume that the linear holonomy of $L_0$ is abelian. Then the holonomy group of the leaf $L_0$ is a finite cyclic  analytically linearizable group.}
\end{Theorem}

This result applies to the study of Riccati foliations in a general setting (cf. \S~\ref{section:Riccati}).

Given a subgroup $G\subset \Diff(\mathbb C^n,0)$ with a finite set of generators $f_1,\ldots,f_r$, by an {\it analytic deformation} of $G$ we shall mean a family $\{G_t\}_{t \in D}$
of subgroups $G_t\subset \Diff(\mathbb C^n,0)$, parametrized by $t \in D\subset \mathbb C$, where each $G_t$ is generated by maps $f_{j,t}\in \Diff(\mathbb C^n,0)$, depending analytically on $t$,  of the form
$f_{j,t}=f_j + \sum\limits_{k=1}^\infty a_{j,k}t^k$ where each $a_{j,k}$ is holomorphic
with a zero of order $\geq 2$ at the origin (i.e., the linear part of $a_{j,k}$ at the origin is zero) for all $j,k$.

Then we can state the following stability theorem for groups of germs of complex diffeomorphisms:

\begin{Theorem}
\label{Theorem:stability}{\rm
Let $G\subset \Diff(\mathbb C^n,0)$ be a cyclic finite subgroup of order $p^{s}$ for some prime number $p$ and $s\in \mathbb N$. Given an analytic deformation $\{G_t\}_{t \in D}$ of $G$ we have the following equivalences:
\begin{enumerate}
\item $G_t$ is irreducible for all $t$ close to $0$.

\item $G_t$ is finite cyclic for all $t$ close to $0$.

\end{enumerate}

Furthermore, if $G$ is trivial then any analytic deformation $\{G_t\}$ of $G$ by irreducible groups $G_t$ is such that $G_t$ is trivial for all $t$ close to $0$.
}
\end{Theorem}

As a sample of how our results apply we give:

\begin{Theorem}
\label{Theorem:Riccati3}{\rm
Let $\fa$ be  the foliation by levels  of a rational function $R
\colon \mathbb P^m \times \mathbb P^n \to \mathbb P^n$.
Assume that the  codimension one
component $\sigma_1\subset \sigma$ of the ramification set $\sigma\subset \mathbb P^m$ of $R$ is  empty or irreducible (not necessarily smooth nor normal crossing type) of degree $p^{s}$ for some prime number $p$ and some $s\in \mathbb N$. Let now $\{\fa_t\}_{t \in \mathbb D}$ be an analytic deformation of $\fa=\fa_0$ by
Riccati foliations on $\mathbb P^m \times \mathbb P^n$ leaving invariant the basis $\mathbb P^m \times \{0\}$ (for some point $0 \in \mathbb P^n$).  Then the global holonomy of  $\fa_t$ is finite cyclic for each $t$ close to $0$. In particular, the leaves of $\fa_t$ are closed in  $(\mathbb P^m \setminus \sigma_1(t)) \times \mathbb P^n$, i.e., $\lim(\fa_t)\subset \sigma_1(t) \times \mathbb P^n,$ for all $t$ close to $0$. If $R$ is the second projection $\mathbb P^m \times \mathbb P^n \to \mathbb P^n$ then $\fa_t$ is analytically conjugate to $\fa$ in $(\mathbb P^m \setminus \sigma(t)) \times \mathbb P^n$.}
\end{Theorem}

We refer to \S\ref{section:Riccati} for the details.
Roughly speaking, the ramification set of $\fa$ is the set $\sigma\subset \mathbb P^m$ of base points $x\in \mathbb P^m$ for which the fiber $\{x\}\times \mathbb P^n$ is not transverse to $\fa$. In the above statement $\sigma(t)$ denotes the ramification set of $\fa_t$.

\section{Dimension $n$}
According to Definition~\ref{Definition:groupgen} we have:

\begin{Definition}[irreducible group]
\label{Definition:irreduciblegroup}
{\rm A subgroup $G \subset \Diff(\mathbb C^n,0)$ is {\it irreducible} if it admits a
finite {\it basic set of generators}
$f_1,f_2,\ldots,f_{\nu +1}\in G$ such that:
\begin{enumerate}
\item[(a)] $f_1\circ f_2\circ\cdots\circ f_{\nu+1}=\Id$.
\item[(b)] $f_i$ and $f_j$ are conjugate in $G$ for all $i,j$.
    \end{enumerate}}

\end{Definition}

In order to prove Theorem~\ref{Theorem:irreducibledimn} we will make use of  the  Taylor expansion. Given $f\in \Diff(\co^n,0)$, since $f(0)=0$, for all $Z\in\co^n$ close to $0$ we have:

\[
f(Z) =  Df(0)\cdot Z + \frac{1}{ 2} f''(0)\cdot Z^2 + \cdots + \frac{1}{p!}f^{(p)}(0)\cdot Z^p + \cdots
\]

where

\[
f^{(p)}(0)\cdot Z^p = \frac{\partial^p f}{\partial Z^p}(0) =  \frac{\partial}{\partial Z}\left(\frac{\partial^{p-1} f}{\partial Z^{p-1}}\right)(0) = \sum^{n}_{k_1,\ldots,k_p=1} \frac{\partial^p f}{\partial Z_{k_{r_{p}}}}(0)z_{k_{1}}\cdots z_{k_{p}}
\]

here $Z=(z_1,\ldots,z_n)$  and $\frac{\partial^p }{\partial Z_{k_{r_{p}}}} = \frac{\partial^p }{\partial{z_{k_1}}\partial{z_{k_2}}\cdots\partial{z_{k_p}}}$. We shall need the following
 expression for the $n$-th derivative of the function composition of  maps:

\begin{Lemma}\label{Leibnitzgeraln}{\rm
For any  $\varphi:\co^n\to\co$, $\psi:\co^n\to\co^n$ holomorphic and $m\in\na$, $m\geq3$ we have that
\begin{multline}\label{hipindn}
   \frac{\partial^m (\varphi\circ\psi)}{\partial{z_{r_m}}\cdots\partial{z_{r_1}}} = \sum\limits_{k_{1},\ldots ,k_m = 1}^{n}\frac{\partial^m
 \varphi(\psi)}{\partial Z_{k_{r_{m}}}}\frac{\partial\psi_{k_m}}{\partial z_{r_m}}\cdots \frac{\partial\psi_{k_1}}{\partial z_{r_1}} +
 \sum\limits_{k= 1}^{n}\frac{\partial\varphi(\psi)}{\partial z_{k}}\frac{\partial^m\psi_k}{\partial{z_{r_m}}\cdots\partial{z_{r_1}}} + \\
  + \sum\limits_{p =2}^{m-1}\sum\limits_{k_{1},\ldots, k_p = 1}^{n}\frac{\partial^p \varphi(\psi)}{\partial Z_{k_{r_p}}}\cdot R_{k_{r_p}}(\cdot)\hspace*{5cm}
\end{multline}

where $\psi_{k_j}$ is the $k_j$-th coordinate of $\psi$ and $R_{k_{r_p}}$ is a polynomial expression as
a function of the derivatives of $\psi$ from order $1$ to order $m-1$ and has no terms containing only derivatives of order $1$.}
\end{Lemma}
\begin{proof}
Applying $\frac{\partial}{\partial z_{r_1}}$ to $\varphi\circ\psi$, we have

\[
\frac{\partial (\varphi\circ\psi)}{\partial{z_{r_1}}} =
\sum\limits_{k = 1}^{n}\frac{\partial \varphi(\psi)}{\partial z_{k}}\frac{\partial\psi_{k}}{\partial z_{r_1}}
\]

applying $\frac{\partial}{\partial z_{r_2}}$ we have

\[
\frac{\partial^2 (\varphi\circ\psi)}{\partial{z_{r_2}}\partial{z_{r_1}}} =
\sum\limits_{k_1,k_2 = 1}^{n}\frac{\partial^2 \varphi(\psi)}{\partial z_{k_2}\partial z_{k_1}}\frac{\partial \psi_{k_2}}{\partial z_{r_2}}
\frac{\partial \psi_{k_1}}{\partial z_{r_1}} + \sum\limits_{k = 1}^{n}\frac{\partial \varphi(\psi)}{\partial z_{k}}
\frac{\partial^2 \psi_{k}}{\partial z_{r_2}\partial z_{r_1}}
\]

applying $\frac{\partial}{\partial z_{r_3}}$ we have

\begin{multline*}
  \frac{\partial^3 (\varphi\circ\psi)}{\partial{z_{r_3}}\partial{z_{r_2}}\partial{z_{r_1}}} =
\sum\limits_{k_1,k_2,k_3 = 1}^{n}\frac{\partial^3 \varphi(\psi)}{\partial z_{k_3}\partial z_{k_2}\partial z_{k_1}}
\frac{\partial \psi_{k_3}}{\partial z_{r_3}}\frac{\partial \psi_{k_2}}{\partial z_{r_2}}\frac{\partial \psi_{k_1}}{\partial z_{r_1}}
+ \sum\limits_{k = 1}^{n}\frac{\partial \varphi(\psi)}{\partial z_{k}}
\frac{\partial^3 \psi_{k}}{\partial z_{r_3}\partial z_{r_2}\partial z_{r_1}} + \\
  + \sum\limits_{k_1,k_2 = 1}^{n}\frac{\partial^2 \varphi(\psi)}{\partial z_{k_2}\partial z_{k_1}}
  \left[\frac{\partial^2 \psi_{k_2}}{\partial z_{r_3}\partial z_{r_2}}\frac{\partial \psi_{k_1}}{\partial z_{r_1}} +
    \frac{\partial^2 \psi_{k_2}}{\partial z_{r_3}\partial z_{r_1}}\frac{\partial \psi_{k_1}}{\partial z_{r_2}} +
    \frac{\partial^2 \psi_{k_2}}{\partial z_{r_2}\partial z_{r_1}}\frac{\partial \psi_{k_1}}{\partial z_{r_3}}\right]
\end{multline*}

then

\[
R_{k_{r_p}}(\cdot) = \frac{\partial^2 \psi_{k_2}}{\partial z_{r_3}\partial z_{r_2}}\frac{\partial \psi_{k_1}}{\partial z_{r_1}} +
    \frac{\partial^2 \psi_{k_2}}{\partial z_{r_3}\partial z_{r_1}}\frac{\partial \psi_{k_1}}{\partial z_{r_2}} +
    \frac{\partial^2 \psi_{k_2}}{\partial z_{r_2}\partial z_{r_1}}\frac{\partial \psi_{k_1}}{\partial z_{r_3}}.
\]

Thus $R_{k_{r_p}}(\cdot)$ is a polynomial expression as
a function of the derivatives of $\psi$ from order 1 and 2 and has no terms containing only derivatives of order 1.\\

Let us assume that equation (\ref{hipindn}) is satisfied for $m$ by showing that it is valid for $m+1$.
By the hypothesis of induction (\ref{hipindn}) is valid. Applying $\frac{\partial}{\partial{z_{r_{m+1}}}}$ to  (\ref{hipindn}) we have

\begin{multline*}
 \frac{\partial^{m + 1} (\varphi\circ\psi)}{\partial{z_{r_{m+1}}}\cdots\partial{z_{r_1}}} = \sum\limits_{k_1,\ldots, k_{m+1} = 1}^{n}\frac{\partial^{m +1}
 \varphi(\psi)}{\partial Z_{k_{r_{m+ 1}}}}\frac{\partial\psi_{k_{m+ 1}}}{\partial z_{r_{m+1}}}\cdots \frac{\partial\psi_{k_1}}{\partial z_{r_1}} +
 \sum\limits_{k= 1}^{n}\frac{\partial\varphi(\psi)}{\partial z_{k}}\frac{\partial^{m+1}\psi_k}{\partial{z_{r_{m+1}}}\cdots\partial{z_{r_1}}} +
  \\
  + \sum\limits_{k_1,\ldots, k_{m} = 1}^{n}\frac{\partial^m
 \varphi(\psi)}{\partial Z_{k_{r_{m}}}}\frac{\partial^2\psi_{k_m}}{\partial z_{r_{m+1}}\partial z_{r_m}}
\frac{\partial\psi_{k_{m-1}}}{\partial z_{r_{m-1}}}\cdots \frac{\partial\psi_{k_1}}{\partial z_{r_1}}
+ \sum\limits_{k_1,\ldots, k_{m} = 1}^{n}\frac{\partial^m
\varphi(\psi)}{\partial Z_{k_{r_{m}}}}\frac{\partial\psi_{k_m}}{\partial z_{r_m}}
\frac{\partial^2\psi_{k_{m-1}}}{\partial z_{r_{m+1}}\partial z_{r_{m-1}}}\cdots \frac{\partial\psi_{k_1}}{\partial z_{r_1}}\\
 +\cdots+
 \sum\limits_{k_1,\ldots, k_{m} = 1}^{n}\frac{\partial^m
 \varphi(\psi)}{\partial Z_{k_{r_{m}}}}\frac{\partial\psi_{k_m}}{\partial z_{r_m}}
 \frac{\partial\psi_{k_{m-1}}}{\partial z_{r_{m-1}}}\cdots \frac{\partial^2\psi_{k_1}}{\partial z_{r_{m+1}}\partial z_{r_1}}
 + \sum\limits_{p = 2}^{m - 1}\sum\limits_{k_1,\ldots, k_{p + 1} = 1}^{n}
 \frac{\partial^{p + 1} \varphi(\psi)}{\partial Z_{k_{r_{p+1}}}}\frac{\partial\psi_{k_{p+ 1}}}{\partial z_{r_{m+1}}}\cdot R_{k_m}(\cdot)\\
 + \sum\limits_{p =2}^{m-1}\sum\limits_{k_{1},\ldots, k_p = 1}^{n}\frac{\partial^p \varphi(\psi)}{\partial Z_{k_{r_p}}}
 \frac{\partial [R_{k_{r_p}}(\cdot)]}{\partial z_{r_{m+1}}} + \sum\limits_{k_1,k_2= 1}^{n}\frac{\partial^2\varphi(\psi)}{\partial z_{k_2}\partial z_{k_1}}\frac{\partial\psi_{k_2}}{\partial z_{r_{m+1}}}\frac{\partial^m\psi_{k_1}}{\partial{z_{r_m}}\cdots\partial{z_{r_1}}}.
  \end{multline*}

Notice that $R_{k_{r_p}}$ is a polynomial expression as
a function of the derivatives of $\psi$ from order 1 to order $m-1$ and has no terms containing only derivatives of order 1. Thus by the chain rule we have that $\frac{\partial [R_{k_{r_p}}(\cdot)]}{\partial z_{r_{m+1}}}$ is is a polynomial expression as
a function of the derivatives of $\psi$ from order 1 to order $m$ and has no terms containing only derivatives of order 1.
Putting  $\frac{\partial^p \varphi(\psi)}{\partial Z_{k_{r_p}}}$ in evidence in the above expression we conclude.
\end{proof}

The very first case in Theorem~\ref{Theorem:irreducibledimn} is the following:

\begin{Proposition}
\label{Proposition:groups n}{\rm
Let $ G\subset \Diff(\mathbb C^n,0)$ be an irreducible group such that $G$ has a generator tangent to identity. Then $ G=\{\Id\}$.}
\end{Proposition}
\begin{proof}
Since $ G$ is an irreducible group that has a generator tangent to identity there exists a finite set of generators $f_{(1)},f_{(2)},\ldots,f_{(\nu +1)}\in  G$ such that
\begin{enumerate}
\item[(a)] $f_{(1)}\circ f_{(2)}\circ\ldots\circ f_{(\nu+1)}=\Id$.
\item[(b)] $f_{(i)}$ and $f_{(j)}$ are conjugate in $ G$ for all $i,j$.
\item[(c)] $Df_{(r)}(0) = \Id$ for some $r$.
    \end{enumerate}
From (b) given $f_{(i)}\in G$ there is $h\in G$, such that $f_{(i)}\circ h(Z) = h\circ f_{(r)}(Z)$. Now from (c) we have
\[
Df_{(i)}(0)Dh(0) = Dh(0)Df_{(r)}(0) = Dh(0)\Id = Dh(0).
\]

Hence $Df_{(i)}(0) = \Id$ for all $i$ and consequently for all $g\in G$, $Dg(0) = \Id$.
Note that if $g(Z) = Z + P_{k}(Z) + h.o.t\in G$ then
\[
g^{(m)}(Z) = g\circ g\circ\cdots\circ g(Z) = Z + mP_k(Z) + h.o.t.
\]

Now for simplicity consider $f = f_{(i)}$ and $g = f_{(j)}$, we will show that
\[
\frac{\partial^m f(0)}{\partial{z_{r_m}}\cdots\partial{z_{r_1}}} = \frac{\partial^m g(0)}{\partial{z_{r_m}}\cdots\partial{z_{r_1}}} = 0\;\mbox{ for all }m\in\mathbb{N},\;m\geq 2.
\]

Indeed, we will prove for each coordinate of $f$ and $g$ using induction.
Now from (b) there is $h\in G$, such that $f\circ h(Z) = h\circ g(Z)$.
We consider $f_s:\co^n\to\co$ the $s$-coordinate of $f$ and $h_s:\co^n\to\co$ the $s$-coordinate of $h$.
Then $f_s\circ h = h_s\circ g$ applying $\frac{\partial}{\partial z_{r_1}}$ on both sides we have

\[
\sum\limits_{k = 1}^{n}\frac{\partial f_s(h)}{\partial z_{k}}\frac{\partial h_{k}}{\partial z_{r_1}} =
\sum\limits_{k = 1}^{n}\frac{\partial h_s(g)}{\partial z_{k}}\frac{\partial g_{k}}{\partial z_{r_1}}
\]

applying $\frac{\partial}{\partial z_{r_2}}$ on both sides we have

\begin{multline*}
  \sum\limits_{k_1,k_2 = 1}^{n}\frac{\partial^2 f_s(h)}{\partial z_{k_2}\partial z_{k_1}}\frac{\partial h_{k_2}}{\partial z_{r_2}}
\frac{\partial h_{k_1}}{\partial z_{r_1}} + \sum\limits_{k = 1}^{n}\frac{\partial f_s(h)}{\partial z_{k}}\frac{\partial^2 h_{k}}{\partial z_{r_2}\partial z_{r_1}} =\\
\sum\limits_{k_1,k_2 = 1}^{n}\frac{\partial^2 h_s(g)}{\partial z_{k_2}\partial z_{k_1}}\frac{\partial g_{k_2}}{\partial z_{r_2}}
\frac{\partial g_{k_1}}{\partial z_{r_1}} + \sum\limits_{k = 1}^{n}\frac{\partial h_s(g)}{\partial z_{k}}\frac{\partial^2 g_{k}}{\partial z_{r_2}\partial z_{r_1}}.
\end{multline*}

As $Df(0) = Dg(0) = Dh(0) = \Id$, we have that
\begin{itemize}
  \item $\frac{\partial f_s}{\partial z_k}(0) = 0$ for all $k\ne s$ and $\frac{\partial f_s}{\partial z_s}(0) = 1$,
  \item $\frac{\partial h_k}{\partial z_r}(0) = 0$ for all $k\ne r$ and $\frac{\partial h_r}{\partial z_r}(0) = 1$,
  \item $\frac{\partial g_k}{\partial z_r}(0) = 0$ for all $k\ne r$ and $\frac{\partial g_r}{\partial z_r}(0) = 1$.
\end{itemize}

Then

\[
\frac{\partial^2 f_s(0)}{\partial z_{r_2}\partial z_{r_1}} + \frac{\partial^2 h_s(0)}{\partial z_{r_2}\partial z_{r_1}} =
\frac{\partial^2 h_s(0)}{\partial z_{r_2}\partial z_{r_1}} + \frac{\partial^2 g_s(0)}{\partial z_{r_2}\partial z_{r_1}}
\]

thus

\[
\frac{\partial^2 f_s(0)}{\partial z_{r_2}\partial z_{r_1}} = \frac{\partial^2 g_s(0)}{\partial z_{r_2}\partial z_{r_1}}.
\]

Now using Taylor theorem we have that
\[
f_{(i)}(Z)= Z + P_2(Z) + h.o.t.,\mbox{ for all }1\leq i\leq \nu+1.
\]

so

\[
f_{(i)}^{(\nu +1)}(Z)= Z + (\nu +1)P_2(Z) + h.o.t.
\]

From (a) we have that $P_2(0) = 0$. Then
\[
\frac{\partial^2 f_s(0)}{\partial z_{r_2}\partial z_{r_1}} = \frac{\partial^2 g_s(0)}{\partial z_{r_2}\partial z_{r_1}} = 0.
\]

Now suppose the statement below is satisfied for $ 3\leq l <m$. We will  be showing that it is valid for $m$. Then

\[
\frac{\partial^l f(0)}{\partial{z_{r_l}}\cdots\partial{z_{r_1}}} = \frac{\partial^l g(0)}{\partial{z_{r_l}}\cdots\partial{z_{r_1}}} = 0.
\]

As $f_s\circ h = h_s\circ g$ we have

\[
\frac{\partial^m (f_s\circ h)}{\partial{z_{r_m}}\cdots\partial{z_{r_1}}} = \frac{\partial^m (h_s\circ g)}{\partial{z_{r_m}}\cdots\partial{z_{r_1}}} .
\]

From Lemma~\ref{Leibnitzgeraln} we have

\begin{equation*}
\sum\limits_{k_{1},\ldots, k_m = 1}^{n}\frac{\partial^m
 f_s(h)}{\partial Z_{k_{r_{m}}}}\frac{\partial h_{k_m}}{\partial z_{r_m}}\cdots \frac{\partial h_{k_1}}{\partial z_{r_1}} +
 \sum\limits_{k= 1}^{n}\frac{\partial f_s(h)}{\partial z_{k}}\frac{\partial^m h_k}{\partial{z_{r_m}}\cdots\partial{z_{r_1}}} +
\sum\limits_{p =2}^{m-1}\sum\limits_{k_{1},\ldots, k_p = 1}^{n}\frac{\partial^{p} f_s(h)}{\partial Z_{k_{r_p}}}\cdot R_{k_{r_p}}(\cdot)
 \end{equation*}

\[
\parallel
\]

\begin{equation*}
\sum\limits_{k_{1},\ldots, k_m = 1}^{n}\frac{\partial^m
 h_s(g)}{\partial Z_{k_{r_{m}}}}\frac{\partial g_{k_m}}{\partial z_{r_m}}\cdots \frac{\partial g_{k_1}}{\partial z_{r_1}} +
 \sum\limits_{k= 1}^{n}\frac{\partial h_s(g)}{\partial z_{k}}\frac{\partial^m g_k}{\partial{z_{r_m}}\cdots\partial{z_{r_1}}} +
\sum\limits_{p =2}^{m-1}\sum\limits_{k_{1},\ldots, k_p = 1}^{n}\frac{\partial^{p} h_s(g)}{\partial Z_{k_{r_p}}}\cdot R_{k_{r_p}}(\cdot).
 \end{equation*}

As $R_{k_{r_p}}$ is a polynomial expression as
a function of the derivatives of $f_s$ (respectively $h_s$)
from order 1 to order $m-1$ and has no terms containing only derivatives of order 1,
by the induction hypothesis when $Z = 0$ we have that $R_{k_{r_p}} (\cdot) = 0$. Now as $Df(0) = Dg(0) = Dh(0) = \Id$, we have that

\[
\frac{\partial^m f_s(0)}{\partial{z_{r_m}}\cdots\partial{z_{r_1}}} +
\frac{\partial^m h_s(0)}{\partial{z_{r_m}}\cdots\partial{z_{r_1}}} =
\frac{\partial^m h_s(0)}{\partial{z_{r_m}}\cdots\partial{z_{r_1}}} +
\frac{\partial^m g_s(0)}{\partial{z_{r_m}}\cdots\partial{z_{r_1}}}
\]

then
\[
\frac{\partial^m f_s(0)}{\partial{z_{r_m}}\cdots\partial{z_{r_1}}} =
\frac{\partial^m g_s(0)}{\partial{z_{r_m}}\cdots\partial{z_{r_1}}}.
\]

Now using Taylor development we have that
\[
f_{(i)}(Z)= Z + P_m(Z) + h.o.t.,\mbox{ for all }1\leq i\leq \nu+1
\]

so

\[
f_{(i)}^{(\nu +1)}(Z)= Z + (\nu +1)P_m(Z) + h.o.t.
\]

From (a) we have that $P_m(0) = 0$. Then
\[
\frac{\partial^m f_s(0)}{\partial{z_{r_m}}\cdots\partial{z_{r_1}}} =
\frac{\partial^m g_s(0)}{\partial{z_{r_m}}\cdots\partial{z_{r_1}}} = 0\;\mbox{ for all }m\in \mathbb{N}.
\]

Consequently
\[
f_{(i)}(Z)= Z,\mbox{ for all }1\leq i\leq \nu+1.
\]
therefore $ G = \{\Id\}$.
\end{proof}

Now let us investigate what happens when the linear part of the diffeomorphism
is different from the identity. We will now show the case where all the diffeomorphisms have the same linear part, for this we will use the definition of resonance (\cite{Ilyashenko-Yakovenko} and \cite{Bracci}):

\begin{Definition}{\rm
A {\it multiplicative resonance} between non zero complex numbers $\lambda_1,\ldots,\lambda_n$ is an identity of form
\begin{equation*}\label{resonance auto}
  \lambda_s = \lambda_1^{m_1}\cdots\lambda_n^{m_n}
\end{equation*}
where $s\in\{1,\ldots,n\}$, $m_1,\ldots,m_n\in\mathbb{N}$ and $m_1 + \cdots + m_n \geq 2$. The vector $M = (m_1,\ldots,m_n)\in \mathbb N^n$ is called
the {\it order} of resonance. For simplicity we can say that {\it $\lambda_s$ is resonant} and that ${\bf\lambda}^M = \lambda_1^{m_1}\cdots\lambda_n^{m_n}$.}
\end{Definition}

We are interested in matrices that have resonant eigenvalues and their relation with polynomial functions:

\begin{Definition}
{\rm
If $\lambda_s$ is a resonant eigenvalue with order of resonance $(m_1,\ldots,m_n)$ we call
\[
z_1^{m_1}\cdots z_n^{m_n}\cdot e_{s}
\]
a {\it resonant monomial}. Here $e_1 = (1,\ldots, 0),\ldots, e_n =(0, \ldots,1)$ defines the canonical basis of $\co^n$.}
\end{Definition}

With these definitions we have:

\begin{Theorem}[{\it Poincaré-Dulac normal form} \cite{Bracci}, \cite{Ilyashenko-Yakovenko}]
\label{Poincare-Dulac} {\rm
Let $f\in \Diff(\co^n,0)$ be  a germ of complex diffeomorphism. If $Df(0)$ is diagonalizable  then $f$ is formally conjugate to a formal series $\hat F= df(0) + \hat P_2 + \ldots$, where the $\hat P_j$  are complex polynomial made only of resonant monomials of $f$. In particular if $df(0)$ has no resonances then $f$ is formally linearizable.}
\end{Theorem}

Now we can finish our proof:
\begin{proof}[Proof of Theorem~\ref{Theorem:irreducibledimn}] In short, there exists  a finite set of generators $f_{(1)},f_{(2)},\ldots,f_{(\nu +1)}\in G$ such that
\begin{enumerate}
\item[(a)] $f_{(1)}\circ f_{(2)}\circ\ldots\circ f_{(\nu+1)}=\Id$.
\item[(b)] $f_{(i)}$ and $f_{(j)}$ are conjugate in $ G$ for all $i,j$.
\item[(c)] All maps $f_{(i)}$ have the same linear part, $Df_{(i)}(0) = A$, for all $i$.
    \end{enumerate}
If $A = \Id$ by Proposition~\ref{Proposition:groups n} $ G = \{\Id\}$ is a finite group. Suppose that $A \ne \Id$ from (a) we have that
$A^{\nu+1} = \Id$. An easy computation with Jordan blocks shows that $A$ is diagonalizable. Let us then assume that $A$ is already in diagonal form with  eigenvalues $\lambda_1,\ldots,\lambda_n$. Then
  $\lambda_j ^{\nu+1}=1$ and the eigenvalues of $A$ are in resonance. Now by Theorem~\ref{Poincare-Dulac} we can write $f_{(1)}$ in the normal form, that is,
there is $\varphi\in \widehat{\Diff}(\co^n,0)$ such that
\begin{equation}\label{decompor51}
\varphi\circ f_{(1)}\circ \varphi^{-1}(Z) = \widetilde{f}_{(1)}(Z) = (\lambda_1 z_1,\ldots,\lambda_n z_n) + \widehat{P}_2(Z) + \widehat{P}_3(Z) + \cdots
\end{equation}

where $\widehat{P}_k(Z)$ contains only resonant monomials with complex coefficients.\\

Now we take $\widetilde{ G} = \varphi\circ G\circ\varphi^{-1}$. Thus $\widetilde{G}$ is isomorphic to $ G$ and satisfies all properties of $ G$. Thus we consider $G$ as $\widetilde{ G}$, then we can write the generators as

$$f_j(Z)=\left(\lambda_1z_1+\dsum_{|Q|=k+1}a_{1,Q}^{(j)}Z^Q+h.o.t.,\ldots, \lambda_nz_n+\dsum_{|Q|=k+1}a_{n,Q}^{(j)}Z^Q+h.o.t.\right)$$

For $j = 1,\ldots,\nu +1$, and $f_1$ have the nonlinear part contains only resonant monomials with complex coefficients. Moreover if $g\in G$ then exists $\ell\in\mathbb{Z}$ such that
\[
Dg(0) = D\varphi(0)A^\ell(D\varphi(0))^{-1}
=\left(
  \begin{array}{cccc}
    \lambda_1^\ell & 0 &\cdots & 0\\
    0 & \lambda_2^\ell & \cdots & 0\\
    0 & \cdots & \ddots & 0\\
    0 & 0 & \cdots & \lambda_n^\ell
  \end{array}
\right).
\]

The idea is to show by a formal algorithm
that $G$ is formally linearizable.  For this, suppose that $G$ has no terms of order $k$,
we will prove that the same is true for the terms of order $k +1$. First, note that given $f, g\in\Diff_{k+1}(\co^n,0)\cap G$ we have
$$f(Z)=\left(a_1z_1+\dsum_{|Q|=k+1}a_{1,Q}Z^Q+h.o.t.,\ldots, a_nz_n+\dsum_{|Q|=k+1}a_{n,Q}Z^Q+h.o.t.\right)$$ and
$$g(Z)=\left(b_1z_1+\dsum_{|Q|=k+1}b_{1,Q}Z^Q+h.o.t.,\ldots, b_nz_n+\dsum_{|Q|=k+1}b_{n,Q}Z^Q+h.o.t.\right)$$

where $Q=(q_1,\ldots,q_n)\in\mathbb{N}^n$ with $|Q|=k+1$, $a_r = \lambda_r^m$ and $b_r = \lambda_r^l$ for all $r = 1,\ldots,n$ e some $m,l\in\mathbb{N}$. So we have

$$\begin{array}{l l}f\circ g(Z)&=\left( a_1b_1z_1+\dsum_{|Q|=k+1}\left[a_1b_{1,Q}+a_{1,Q}B^Q\right]Z^Q+h.o.t.\right.,\ldots,\\&\\&\;\;\;\;\;\;\left.a_nb_nz_n+
\dsum_{|Q|=k+1}\left[a_nb_{n,Q}+a_{n,Q}B^Q\right]Z^Q+h.o.t.\right)\end{array}$$

where $B^Q = b_1^{q_1}\cdots b_n^{q_n}$. We will study two cases:

\begin{Claim}{\rm
Let $Q\in \mathbb N^n$ be a order of resonance for some $\lambda_r$, $|Q| = k + 1$. Then $a_{r,Q}^{(j)} = 0$.}
\end{Claim}
\begin{proof}[Proof of the claim]
As $Q\in\mathbb{N}^n$ is a order of resonance for $\lambda_r$, we have that $\lambda^Q = \lambda_r$. Then $B^Q = (\lambda^Q)^l = b_r$. Now
we define the following application

$$\begin{array}{c c c c} \varphi_{r,Q}:&G&\to& \co \\ &\left(a_1z_1+\dsum_{|Q|=k+1}a_{1,Q}Z^Q+h.o.t.,\ldots,a_nz_n+\dsum_{|Q|=k+1}a_{n,Q}Z^Q+h.o.t. \right)&\mapsto&\dfrac{a_{r,Q}}{a_r} \end{array}$$
which defines a morphism between $(G,\circ)$ and $(\co,+)$, indeed
$$\varphi_{r,Q}(f\circ g)=\dfrac{a_rb_{r,Q}+a_{r,Q}B^Q}{a_rb_r} = \dfrac{a_rb_{r,Q}+a_{r,Q}b_r}{a_rb_r} = \varphi_{r,Q}(f)+\varphi_{r,Q}(g)$$
Now for $i\neq j$ from (b) we have $h\in G$ such that
$$f_i\circ h=h\circ f_j$$
applying the morphism we have
$$\begin{array}{c} \varphi_{r,Q}(f_i\circ h)=\varphi_{r,Q}(h\circ f_j)\\ \\
\varphi_{r,Q}(f_i)+\varphi_{r,Q}(h)=\varphi_{r,Q}(h)+\varphi_{r,Q}(f_j)\\ \\
\varphi_{r,Q}(f_i)=\varphi_{r,Q}(f_j)\\
\\
\dfrac{a^{(i)}_{r,Q}}{\lambda_r}=\dfrac{a^{(j)}_{r,Q}}{\lambda_r}
\\
\\
a^{(i)}_{r,Q}=a^{(j)}_{r,Q}.
  \end{array}$$
  Lets denote $a_{r,Q}:=a^{(1)}_{r,Q}=\ldots=a^{(\nu+1)}_{r,Q}$.  Now from (a) we have
$$\begin{array}{c} \varphi_{r,Q}(f_1\circ f_2\circ\ldots\circ f_{\nu+1})=\varphi_{r,Q}(\Id)=0\\ \\
\varphi_{r,Q}(f_1)+\varphi_{r,Q}(f_2)+\ldots+\varphi_{r,Q}(f_{\nu+1})=0\\ \\
\dfrac{a^{(1)}_{r,Q}}{\lambda_r}+\dfrac{a^{(2)}_{r,Q}}{\lambda_r}+\ldots+\dfrac{a^{(\nu+1)}_{r,Q}}{\lambda_r}=0\\
\\
(\nu+1)\left(\dfrac{a_{r,Q}}{\lambda_r}\right)=0.
  \end{array}$$
  Therefore $a_{r,Q}=0$.
\end{proof}
\begin{Claim}{\rm
 If $Q$ is not the order of resonance of any $\lambda_r$, $|Q| = k +1$  then $a_{r,Q}^{(j)} = a_{r,Q}^{(1)} = 0$.}
\end{Claim}
\begin{proof}[Proof of the claim]
Note that $f_1$ is written in its normal form, then $a_{r,Q}^{(1)} = 0$. Now we define the following application
$\psi_Q\colon G\to \Aff(\co)^n$,

\[
\psi_Q\left(a_1z_1+\dsum_{|Q|=k+1}a_{1,Q}Z^Q+h.o.t.,\ldots,a_nz_n+\dsum_{|Q|=k+1}a_{n,Q}Z^Q+h.o.t. \right)=
\]
\[
\left(\dfrac{a_1z_1+a_{1,Q}}{A^Q},\ldots,\dfrac{a_nz_n+a_{n,Q}}{A^Q}\right)
\]
where $A^Q = a_1^{q_1}\cdots a_n^{q_n}$, which defines a
morphism from  $(G,\circ)$ into $(\Aff(\co)^n,\circ)$. Indeed,
$$\psi_Q(f\circ g)=\left(\dfrac{a_1b_1z_1+a_1b_{1,Q}+a_{1,Q}B^Q}{A^QB^Q},\ldots,\dfrac{a_nb_nz_n+a_nb_{n,Q}+a_{n,Q}B^Q}{A^QB^Q}\right)$$
on the other hand
$$\begin{array}{r c l} \psi_Q(f)\circ\psi_Q(g)&=&\left(\dfrac{a_1\left(\dfrac{b_1z_1+b_{1,Q}}{B^Q}\right)+a_{1,Q}}{A^Q},\ldots,\dfrac{a_n\left(\dfrac{b_nz_n+b_{n,Q}}{B^Q}\right)
+a_{n,Q}}{A^Q}\right)\\ &&\\&=&\left(\dfrac{a_1b_1z_1+a_1b_{1,Q}+a_{1,Q}B^Q}{A^QB^Q},\ldots,\dfrac{a_nb_nz_n+a_nb_{n,Q}+a_{n,Q}B^Q}{A^QB^Q}\right)\\ &&\\&=&\psi_Q(f\circ g)\end{array}$$

Denote by $G_Q$ the image of $G$ by $\psi_Q$. Therefore $G_Q$  is an product affine group generated by the transformations $g^Q_{(1)},g^Q_{(2)},\ldots,g^Q_{(\nu+1)}$, pairwise  conjugate in $G_Q$, where
$$g^Q_{(i)}(Z)=\left(\dfrac{\lambda_1 z_1+a^{(i)}_{1,Q}}{\lambda_1^{q_1}\cdots \lambda_n^{q_n}},\ldots,\dfrac{\lambda_n z_n+a^{(i)}_{n,Q}}{\lambda_1^{q_1}\cdots \lambda_n^{q_n}}\right)=(g^Q_{i,1}(z_1),\ldots, g^Q_{i,n}(z_n)).$$

Denote by $\lambda^Q = \lambda_1^{q_1}\cdots \lambda_n^{q_n}$ and $G^Q_r$ is an affine group generated by the transformations $g^Q_{1,r},\ldots, g^Q_{\nu+1,r}$, pairwise conjugate in $G^Q_r$, where
$$g^Q_{i,r}(w)=\dfrac{\lambda_r w+a^{(i)}_{r,Q}}{\lambda^Q}.$$

We now apply  the following lemma whose proof is found in \cite{CL}
(page 222):

\begin{Lemma}\label{keylemma}{\rm Let $\eta$ be  a $l$-th root of the unit, $l>1$,
 $\beta_1,\beta_2,\ldots,\beta_{r+1}\in\co$ and $\Gamma$
 an affine group generated by the transformations
  $h_i(z)=\eta z+\beta_i$, $i=1,2,\ldots,r+1$.
Then the $h_i$'s are pairwise conjugate in $\Gamma$ if and  only if
either $l$ has two distinct prime divisors or $l=q^m$, for some
prime $q$ and some  $m\in\mathbb{N}^*$ and
$\beta_1=\beta_2=\ldots=\beta_{r+1}$.}
\end{Lemma}

Taking $\eta=\dfrac{\lambda_r}{\lambda^Q}$, $l=p^{s_1+\ldots+s_n}$, $r=\nu$  and
$\beta_i=\dfrac{a^{(i)}_{r,Q}}{\lambda^Q}$ ($i=1,\ldots,\nu+1$) by the Lemma \ref{keylemma} we have $p^{s_1+\ldots+s_n}=q^m$, $q$ prime, $m\in\mathbb{N}^*$ and
$$\dfrac{a^{(1)}_{r,Q}}{\lambda^Q}=\ldots=\dfrac{a^{(\nu+1)}_{r,Q}}{\lambda^Q}.$$
Therefore $q=p$, $m=s_1+\ldots+s_n$ and
$$a^{(1)}_{r,Q}=\ldots=a^{(\nu+1)}_{r,Q}.$$

\end{proof}
This produces a convergent algorithm in the Krull topology. This already shows that $G$ is formally linearizable. Since there is a set of generators all with the same linear part $A$ which is a finite order matrix, this implies that $G$ is cyclic generated by $A$. The proof of Theorem~\ref{Theorem:irreducibledimn} is now complete.
\end{proof}

\begin{proof}[Proof of Corollary~\ref{Corollary:irreducibledimn}]
Since the linear part of $G$ is abelian, given a basic set of   generators $f_j$ of $G$ in the definition of irreducible group, all the maps $f_j$ have the same linear part, say $A\in \GL(n,\mathbb C)$. In particular $A$ satisfies $A^{\nu+1}=\Id$. On the other hand, there is a set of generators $g_1,\ldots,g_{p^s}$ with $g_1\circ \cdots\circ g_{p^s}=\Id$. This implies that $A^{p^s}=\Id$. Since $p$ is prime this implies that the order of $A$ is $p^r$ for some $r\in \{0,\ldots,s\}$ and therefore $A$ satisfies the hypothesis of Theorem~\ref{Theorem:irreducibledimn}.
The group $G$ is therefore finite and cyclic.
\end{proof}

The hypothesis  that the eigenvalues of the linear part are roots of the unit of order  power of a {\em same prime  number}  cannot be dropped,  as shown in the following examples.

\begin{Example}{\rm  Let $G\subset \Diff(\mathbb{C}^2,0)$ be the subgroup generated
by the maps $f_1(Z)=f_2(Z)=f_3(Z)=f_4(Z)=(-z_1,\lambda z_2), f_5(Z)=(-z_1+z_2^2,\lambda z_2)$ and $f_6(Z)=(-z_1+\lambda^2z_2^2,\lambda z_2)$ where $\lambda^3=1$ so that $\lambda^2+\lambda+1=0$. Note that the generators have the same linear part with eigenvalues roots of order $2$ and $3$. We claim  that $G$ is irreducible and not finite (not linearizable). The first condition is satisfied
$$\begin{array}{r c l}f_1\circ f_2\circ f_3\circ
f_4\circ f_5\circ f_6(Z)&=&f_1\circ f_2\circ f_3\circ
f_4\circ f_5\left(-z_1+\lambda^2z_2^2,\lambda z_2\right)\\ &&\\&=&f_1\circ f_2\circ f_3\circ
f_4(z_1-\lambda^2z_2^2+\lambda^2z_2^2,\lambda^2z_2)\\&&\\&=&f_1\circ f_2\circ f_3\circ
f_4(z_1,\lambda^2z_2)=(z_1,\lambda^6z_2)=Z
 \end{array}$$
 To check the second condition take  $g_1(Z)=f_1^4\circ f_5\circ
f_1(Z)=(z_1+\lambda^2z_2^2,z_2)$ then  $g_1\in G$ and note that

$$f_1\circ g_1(Z)=f_1(z_1+\lambda^2z_2^2,z_2)=(-z_1-\lambda^2z_2^2,\lambda z_2) $$
$$g_1\circ f_5(Z)=g_1(-z_1+z_2^2,\lambda z_2)=(-z_1+z_2^2+\lambda^2(\lambda^2z_2^2),\lambda z_2)=(-z_1+(1+\lambda)z_2^2,\lambda z_2)$$ Since
$\lambda^2+\lambda+1=0$ we have $1+\lambda=-\lambda^2$ therefore $f_1\circ
g_1=g_1\circ f_5$.

Take $g_2(Z)=f_5\circ f_1^5(Z)=(z_1+\lambda z_2^2, z_2)$ then $g_2\in G$ and note that

$$f_1\circ g_2(Z)=f_1(z_1+\lambda z_2^2, z_2)=(-z_1-\lambda z_2^2,\lambda z_2)$$
$$g_2\circ f_6(Z)=g_2\left(-z_1+\lambda^2 z_2^2,\lambda z_2\right)=(-z_1+\lambda^2z_2^2+\lambda(\lambda^2z_2^2),\lambda z_2)=(-z_1+(\lambda^2+1)z_2^2,\lambda z_2)$$
Since
$\lambda^2+\lambda+1=0$ we have $1+\lambda^2=-\lambda$ therefore $f_1\circ g_2=g_2\circ f_6$.

Take $g_3(Z)=f_5^2\circ f_1^4(Z)=(z_1+(\lambda-\lambda^2)z_2^2,z_2)$ then $g_3\in G$ and note that

$$f_5\circ g_3(Z)=f_5\left(z_1+(\lambda-\lambda^2)z_2^2,z_2\right)=(-z_1+(\lambda^2-\lambda)z_2^2+z_2^2,\lambda z_2)=(-z_1+(\lambda^2-\lambda+1)z_2^2,\lambda z_2)$$
$$g_3\circ f_6(Z)=g_3\left(-z_1+\lambda^2z_2^2,\lambda z_2\right)=(-z_1+\lambda^2 z_2^2-(\lambda^2-\lambda)(\lambda^2z_2^2),\lambda z_2)=(-z_1+(\lambda^2-\lambda+1)z_2^2,\lambda z_2)$$
therefore $f_5\circ g_3=g_3\circ f_6$.

Consequently the $f_j$ are pairwise conjugate in the group $G$. Note that also $g_1^n(Z)=(z_1+n\lambda^2z_2^2,z_2)\neq Z$ for all $n\in\mathbb{N}^*$. Therefore $G$ it is not finite. Finally, we observe that $G$ is not abelian and in particular, it is not analytically linearizable: indeed, if $g$ linearizes $G$,  then $g\circ f_1\circ g^{-1}= f_1$. On the other hand $g\circ f_5\circ
g^{-1}=f_1$, so $$f_5=g^{-1}\circ f_1\circ g = f_1$$ which is a contradiction.}
\end{Example}

\begin{Example} {\rm Consider $G\subset \Diff(\mathbb{C}^2,0)$ the subgroup generated
by the maps
$$f_1(Z)=\ldots=f_{10}(Z)=(iz_1,\lambda z_2), f_{11}(Z)=(iz_1,\lambda z_2+z_1^2)\;\;\mbox{ and }\;\;f_{12}(Z)=(iz_1,\lambda z_2+\lambda^2 z_1^2)$$where $\lambda^3=1$ so that $\lambda^2+\lambda+1=0$.
Note that the generators have the same linear part with eigenvalues roots of order 4 and 3. $G$ is irreducible and not finite (not linearizable).}
\end{Example}

\begin{Example}{\rm Consider $G\subset \Diff(\mathbb{C}^2,0)$ the subgroup generated
by the maps
$$f_1(Z)=\ldots=f_{34}(Z)=(-z_1,\lambda z_2), f_{35}(Z)=(-z_1+z_2^3,\lambda z_2)\;\;\mbox{ and }\;\;f_{36}(Z)=(-z_1+\lambda^3 z_2^3,\lambda z_2)$$where $\lambda^9=1$ so that $\lambda^6+\lambda^3+1=0$.
Note that the generators have the same linear part with eigenvalues roots of order 4 and 9. $G$ is irreducible and not finite (not linearizable).}
\end{Example}

\section{Applications}

As a first application we prove
\begin{proof}[Proof of Theorem~\ref{Theorem:leafholonomy}]
By hypothesis the linear part of the holonomy group
$\Hol(L_0)\hookrightarrow \Diff(\mathbb C^n,0)$ is abelian.
By Deligne's theorem  the
 fundamental group $\pi_1(L_0)$ is irreducible. Indeed, it is generated by
 a small simple loop $\gamma$ around $C$ and its conjugagy homotopy classes $\gamma_j, j \in J$. Choose
 a linear embedding $\ell \colon \mathbb P^1 \to \mathbb P^2$ in general position with respect to $C$. This means that $\ell(\mathbb P^1)$ meets $C$ transversely and only at nonsingular points. In particular the intersection $\ell(\mathbb P^1)\cap C$ is a set of
 $\nu+1=p^s$ points say $\{p_1,\ldots,p_{\nu+1}\}$. Given a base point $p_0\in \ell^{-1}(\mathbb P^2\setminus\{p_1,\ldots,p_{\nu+1}\})$ by Lefechetz hyperplane section theorem, there is a surjective morphism $\pi_1(\ell^{-1}(\mathbb P^2\setminus\{p_1,\ldots,p_{\nu+1}\}),p_0) \to \pi_1(\mathbb P^2 \setminus C)$. Thus we may take the small loop $\gamma=\gamma_1$ contained in a small disc in $\ell(\mathbb P^1)$ centered at $p_1$ and the other homotopy classes as given by small loops $\gamma_j$ contained in small discs in $\ell(\mathbb P^1)$ and centered at the points $p_j, j =2,\ldots,\nu+1$. In particular, $\pi(L_0)$ is irreducible with a set of generators $[\gamma_1],\ldots,[\gamma_{\nu+1}]$ as in Definition~\ref{Definition:groupgen}.
 The corresponding holonomy maps $f_{[\gamma_j]}\in \Hol(\fa, L_0)$ are such that
 $f_{[\gamma_1]},\ldots,f_{[\gamma_{\nu+1}]}$ is a set of generators for  holonomy
group as an irreducible  subgroup of $\Diff(\mathbb C^n,0)$. By hypothesis this group has  an abelian linear part. Since $\nu+1=p^s$,  by Theorem~\ref{Theorem:irreducibledimn} this holonomy group is finite.

\end{proof}

\begin{proof}[Proof of Theorem~\ref{Theorem:stability}]
If $G_t$ is finite and cyclic then it is irreducible. Thus we shall prove that (1) implies (2). Assume that $G_t$ is irreducible for all $t$ close to $0$. By hypothesis $G_t$ is generated by
the maps $f_{j,t}$ above. If $G$ is trivial then clearly any map $f_{j,t}$ is tangent to the identity. In this case, by Theorem~\ref{Theorem:irreducibledimn} the group $G$ is also trivial. Assume now that $s>0$. Since $G=G_0$ is cyclic of order $p^s$, any set of non-trivial generators $\{f_j,j=1,\ldots,r\}$ is of the form
$f_j=f^{n_j}$ for some $n_j \in \{1,\ldots,p^s-1\}$, where $f$ is a generator of $G$ as a cyclic group.
Thus we have $f_{j,t}^\prime(0)=(f^\prime(0))^{n_j}$ and therefore the linear part of the group $G_t$ satisfies the conditions of Theorem~\ref{Theorem:irreducibledimn}. By this same theorem the group $G_t$ is finite cyclic.
\end{proof}

\begin{Definition}
{\rm
Given a map germ $f \in \Diff(\mathbb C^n,0)$ and a hypersurface germ ${H} \subset \mathbb C^n$ through the origin $ 0 \in \mathbb C^n$  we say that ${H}$  is {\it f-invariant at order $k\in \mathbb N$}   if:
 \begin{enumerate}
 \item $f(H)\subset H$.
  \item We have $f^{k}\big|_{H}=\Id$.
  \end{enumerate}

 We shall also say that $H$  is {\it infinitesimally f-invariant at order
  $k\in \mathbb N$ } if:
 \begin{enumerate}
 \item the tangent space $T_0({H})\subset \mathbb C^n_0$  is invariant by the derivative  $f^\prime(0)$, i.e.,
$f^\prime(0) \cdot T_0({H}) =T_0({H})\subset \mathbb C^n_0$.
 \item We have $f^{k}\big|_{H}=\Id$.
  \end{enumerate}
  }

  \end{Definition}

Clearly, if $H$ is $f$-invariant at order $k$ then it is $f$-infinitesimally invariant at order $k$.

\begin{Corollary}{\rm Let $ G\subset \Diff(\mathbb C^n,0)$ be an irreducible group, $p\in \mathbb N$ a prime number. Assume that there are analytic hypersurface germs ${H}_1,\ldots,{H}_n\subset \mathbb C^n$ meeting transversely at the origin such that each $H_j$ is  infinitesimally invariant at order $p^{s_j}$ by each element of $G$. Then $ G$ is a finite group.}
\end{Corollary}
 \begin{proof}
 Up to a change of coordinates we may assume that $H_j : \{z_j=0\}, \, j=1,\ldots,n$.
 Thus, $G$ admits a finite set of generators $f_1,f_2,\ldots,f_{\nu+1}\in  G$ such that:
\begin{enumerate}
\item[(a)] $f_1\circ f_2\circ\cdots\circ f_{\nu+1}=\Id$.
\item[(b)] $f_i$ and $f_j$ are conjugate in $ G$ for all $i,j$.
\end{enumerate}
By hypothesis for each $i,j$ we have $f_j^\prime(0)\cdot T_0(H_i)\subset T_0(H_i)$. This implies
 that
 \begin{enumerate}
\item[(c)] For each $j=1,\ldots,n$ we have $$f_j(Z)= T_jZ + P_{j2}(Z) +\cdots +P_{jk}(Z) + \cdots$$

where $P_{jk}$ is homogeneous of degree $k\geq 2$ and

\[
T_j=Df_j(0)
=\left(
  \begin{array}{cccc}
   \lambda_{j1} & 0 &\cdots & 0\\
    0 & \lambda_{j2} & \cdots & 0\\
   0 & \cdots & \ddots & 0\\
  0 & 0 & \cdots & \lambda_{jn}
  \end{array}
\right).
\]
\end{enumerate}

Also by hypothesis we have $f_j^{p^{s_i}}\big|_{H_i}=\Id$, for all $j=1,\ldots,\nu+1$. This implies
\begin{enumerate}
\item[(d)] $\lambda_{ji}^{p^{s_i}}=1$ for each $j\in \{1,\ldots,\nu+1\}$ and each $i \in \{1,\ldots,n\}$.
    \end{enumerate}
From (b) we have that for $i\neq j$ there exists $g\in G $ such that $$f_i\circ g=g\circ f_j$$ hence we obtain $$Df_i(0)Dg(0) = Dg(0)Df_j(0).$$
Since the generators $f_j$ have a diagonal linear part in the chosen coordinates, the same holds for any element of $G$. Hence $Dg(0)$ is a diagonal matrix. Then $Df_i(0) = Df_j(0)$ for all  $i,j$ and their eigenvalues can be listed as $\lambda_1,\ldots,\lambda_n$ where  $\lambda_j$ is a $p^{s_j}$-th root of the unit, with $p$ prime and $s_j\in \mathbb{N}$.
Therefore by Theorem~\ref{Theorem:irreducibledimn} $ G$ is a finite group.
\end{proof}

\section{Riccati foliations}
\label{section:Riccati}
Let $\pi \colon  E \to B$ be a holomorphic fiber space with fiber $F$. A holomorphic foliation $\fa$ on $E$ with singular set $\sing(\fa)\subset E$ is a {\it Riccati foliation} if there is a subset $\sigma \subset B$ such that:
\begin{enumerate}
\item $\pi^{-1}(\sigma)\subset E$ is a union of invariant fibers.
\item $\fa\big|_{E\setminus \pi^{-1}(\sigma)}$ is transverse to the fibers of the
fiber space $\pi\big|_{E\setminus \pi^{-1}(\sigma)}\colon E\setminus \pi^{-1}(\sigma) \to B\setminus\sigma$ in the sense of Ehresmann (\cite{C-LN} Chapter V).
\end{enumerate}

In particular, we have:
\begin{itemize}
\item[(3)] $\dim E= \dim F + \dim B$ and $\dim \fa= \dim B$;

\item[(4)] $\sing(\fa) \subset \pi^{-1}(\sigma)$.
\end{itemize}
The set $\sigma\subset B$ is called  {\it ramification set} of $\fa$.
Since the restriction $\fa\big|_{E\setminus \pi^{-1}(\sigma)}$ is a foliation transverse to the fibers of the fiber space $\pi \colon E\setminus \pi^{-1}(\sigma) \to B \setminus \sigma$ in the ordinary Ehresmann sense, it is completely described by its {\it global holonomy} (\cite{C-LN} Chapter V). This is a lifting paths homomorphism $\phi \colon \pi_1(B \setminus \sigma) \to
\Aut(F)$.
The very basic example is given by the compactification of the foliation $\fa$ on $\mathbb P^1 \times \mathbb P^1$ given in affine coordinates $(x,y)\in  \mathbb C \times \mathbb C$ by a {\it Riccati differential equation} $\frac{dy}{dx} = \frac{a(x) y^2 + b(x) y + c(x)}{p(x)}$ where the coefficients $a,b,c,p$ are polynomials. In this case the fiber space structure is given by the product and the projection
$\pi(x,y)=x$. Using this notion  a {\it Riccati foliation} on  $\mathbb P^m \times \mathbb P^n$ is a codimension $n$ holomorphic foliation with singularities, such  that for some analytic codimension $\geq $ one subset $\sigma\subset \mathbb P^m$, the foliation is transverse to the vertical fibers $\{x\} \times \mathbb P^n, \, x \in \mathbb P^m \setminus \sigma$ while $\sigma \times \mathbb P^n$ is a union of invariant fibers.
A Riccati foliation will be called a {\it Bernoulli foliation} if there is an invariant horizontal fiber $\mathbb P^m \times \{q\}$, for some $q\in \mathbb P^n$.

We investigate the connection between the geometry of the ramification set with the dynamics of a given Riccati foliation.
We first we make a basic remark: if the ramification set is empty (or, more generally if it has codimension $\geq 2$) then $\pi_1(\mathbb P^m\setminus \sigma)=\{0\}$. This implies that all leaves are compact diffeomorphic to $\mathbb P^m$ and the foliation is equivalent to the second projection
$\mathbb P^n \times \mathbb P^m \to \mathbb P^n, \, (x,y)\mapsto y$. Thus we shall assume
that $\sigma\subset \mathbb P^n$ is nonempty of codimension one.

\subsection{Case $n=m=1$}
Let us begin with the dimension two case. More precisely we consider  the case where $\fa$ is a Riccati foliation in $\mathbb P^1 \times \mathbb P^1$,
assuming that $\fa$ has an irreducible ramification set $\sigma\subset \mathbb P^1$. This implies that $\sigma$ is a single point and we may assume that in affine coordinates $(x,y)$ the ramification point is the point $x=\infty, y=0$. Then we may write $\fa$ as given by
a polynomial differential equation $\frac{dy}{dx}=a(x)y^2 + b(x)y + c(x)$.
The global holonomy of $\fa$ is given by an homomorphism $\phi\colon \pi(\mathbb P^1\setminus \sigma) \to \Aut(\mathbb P^1)$. Since $\sigma$ is a single point we have $\mathbb P^1\setminus \sigma =\mathbb C$ is simply-connected and therefore the global  holonomy is trivial.
By the classification of foliations transverse to fibrations (\cite{C-LN} Chapter V) there is a
fibered biholomorphic map $\Phi\colon  \mathbb C \times \mathbb P^1 \to \mathbb C \times \mathbb P^1$ that takes the foliation $\fa$ into the foliation $\mathcal H $ given by
the horizontal fibers $\mathbb C \times \{y\}, y \in \mathbb P^1$.

\begin{Lemma}
\label{Lemma:aut}{\rm
A holomorphic diffeomorphism $\Phi \colon \mathbb C \times \mathbb P^1 \to \mathbb C \times \mathbb P^1$ preserving the vertical fibration  writes  in affine coordinates $(x,y)\in \mathbb C^2 \subset \mathbb C \times \mathbb P^1$ as $\Phi(x,y)=\left( Ax+B, \frac{a(x)y+ b(x)}{c(x)y + d(x)}\right)$ where $a,b,c,d $ are entire functions satisfying $ad-bc=1$, $0 \ne A,B \in \mathbb C$.}
\end{Lemma}
\begin{proof}[Proof of Lemma~\ref{Lemma:aut}]
 Picard's theorem and the fact that $\Phi$ preserves the fibration $x=const$ show that it is of the form $\Phi(x,y)=(f(x),g(x,y))$ where $f(x)=Ax+B$ is an affine map.  Finally, for each fixed $x\in \mathbb C$ the map $\mathbb P^1 \ni y \mapsto g(x,y) \in \mathbb P^1$ is a diffeomorphism so it must write as $g(x,y)=\frac{a(x)y +b(x)}{c(x)y + d(x)}$ for some entire functions $a,b,c,d$ satisfying $ad - bc=1$.
\end{proof}

In particular we conclude that the leaves of  $\fa$ are diffeomorphic with $\mathbb C$ (including the one contained in the invariant fiber $\{(0,\infty)\}\times \mathbb P^1$, and $\fa$ admits a holomorphic first integral $g\colon \mathbb C \times \mathbb P^1 \to \mathbb P^1$ of the above form $g(x,y)=\frac{a(x)y +b(x)}{c(x)y + d(x)}$.

\subsection{Case $m=2, n=1$}

Assume now that $\fa$ is a codimension one Riccati foliation in $\mathbb P^2 \times \mathbb P^1$. If the codimension one component $\sigma_1\subset \sigma$ of the ramification set $\sigma\subset \mathbb P^2$ is irreducible, smooth or with (double ordinary) normal crossings, then the fundamental group $\pi_1(\mathbb P^2 \setminus \sigma)$ is finite cyclic of order $\deg (\sigma_1)$ (Zariski-Fulton-Deligne). In this case the global holonomy of $\fa$ is a finite cyclic subgroup of $\Aut(\mathbb P^1)$ which corresponds to one of the following possibilities:

\begin{enumerate}
\item A cyclic subgroup generated by a map of the form $z \mapsto \xi z$ where $\xi$ is a root of the unit of order $k \leq \deg \sigma$.

 \item The group generated by the inversion $f(z)=\frac{1}{z}$.
 \end{enumerate}

     Assume that we are in case (1) above. This gives a function $z\mapsto z^k$ in the fiber $\mathbb P^1$ which admits a holonomy extension to $(\mathbb P^2 \setminus \sigma) \times \mathbb P^1 \to \mathbb P^1$ which is constant along the leaves of $\fa$ in
$(\mathbb P^2 \setminus \sigma )\times \mathbb P^1$. This shows that $\fa$ admits  a holomorphic first integral $\vr\colon (\mathbb P^2 \setminus \sigma )\times \mathbb P^1 \to \mathbb P^1$.

Assume now that we are in case (2). In this case we can take the holonomy invariant function $z\mapsto (\ln z)^2$ and extend it to a Liouvillian first integral $\vr$ for $\fa$ in
$\mathbb P^2 \setminus \sigma$.

For a different framework we shall need the remark below:

\begin{Remark}
{\rm
  It is well-known that the group of automorphisms $\Aut(\mathbb P^n)$ is  the projectivization of the linear group $\GL(n+1,\mathbb C)$ of non-singular linear maps of $\mathbb C^{n+1}$ and therefore isomorphic to $\mathbb PSL(n,\mathbb C)$.
}
\end{Remark}

Next we consider another situation:

\begin{Theorem}
\label{Theorem:Riccati1}{\rm
Let $\fa$ be a {\it Bernoulli}  foliation on $\mathbb P^2 \times \mathbb P^1$. Assume that the ramification set $\sigma\subset \mathbb P^2$ is irreducible (not necessarily smooth nor normal crossing type) of degree   $p^{s}$ for some prime number $p$ and some $s\in \mathbb N$.  Then the global holonomy of  $\fa$ is finite cyclic. In particular, the leaves of $\fa$ are closed in  $(\mathbb P^2 \setminus \sigma) \times \mathbb P^1$, i.e., $\lim(\fa)\subset \sigma \times \mathbb P^1$. Moreover, $\fa$ admits a holomorphic first integral
$\vr \colon (\mathbb P^2 \setminus \sigma)\times \mathbb P^1\to \mathbb P^1$.}
\end{Theorem}

\begin{proof}
The global holonomy identifies with a subgroup $H\subset \Aut(\mathbb P^1)$.
  Since $\sigma\subset \mathbb P^2$ is irreducible, $H$ is irreducible.
  Since $\deg(\sigma)=p^s$ and $\sigma$ is irreducible it follows from the same ideas in the proof of Theorem~\ref{Theorem:leafholonomy} that $H$ admits a basic set of generators
  of the form $\{f_1,\ldots,f_{\nu+1}\}\subset \Aut(\mathbb P^1)$ with $\nu+1=p^s$. By hypothesis $\fa$ has an invariant horizontal fiber say $\mathbb P^2 \times \{q\}$. This implies that $H$ has a fixed point at $\{q\}$. We denote by $ H(q)\subset \Diff(\mathbb C^1,0)$ the subgroup induced by the germs at $q$ of maps $h \in H$ (we may identify $q=0$). This group is irreducible and has a basic set of generators consisting of the germs $f_{j,q}$ at $q$ of the maps $f_j, j=1,\ldots,\nu+1=p^s$. By Theorem~\ref{Theorem:irreducibledimn} for dimension $n=1$ this implies that $H(q)$ is finite. In particular $H(q)$ is abelian and each map $f_{j,q}$ has finite order. By the identity principle the maps  $f_j$ commute and have finite order. This implies that $H$ is finite cyclic analytically conjugate in $\mathbb P^1$ to the cyclic group generated by $z\mapsto e^{2\pi i /k}  z$ for some $k \in \mathbb N^*$. As above we can extend the function $z^k$ as a holomorphic  first integral $\vr \colon (\mathbb P^2 \setminus \sigma)\times \mathbb P^1\to \mathbb P^1$ for $\fa$.

Now we proceed. Given a leaf $L$ of $\fa$ not contained in $\pi^{-1}(\sigma)$ we claim that the closure $\ov L\subset \mathbb P^2 \times \mathbb P^1$ is contained in $\pi^{-1}(\sigma)$.
 Indeed, given a generic point $p \in \mathbb P^2 \setminus \sigma$ the fiber $F_p:=\pi^{-1}(p)=\{p\}\times \mathbb P^1$ is transverse to $\fa$. Let us prove that the intersection $L\cap F_p$ is a discrete set. Given two points $z_1, z_2 \in F_p \cap L$ we choose a path $\gamma\subset L$ joining $z_1$ to $z_2$ and project this path into a path $\gamma_0\subset \mathbb P^2 \setminus \sigma$ (recall that $\pi^{-1}(\sigma)$ is invariant). The path $\gamma_0$ is closed based at $p$.
The corresponding global holonomy map $h_{\gamma_0}$ to $\gamma_0$ is such that $h_{\gamma_0}(z_1)=z_2$.
Since the global holonomy group $H$ is finite and cyclic this implies that $\# (F_p \cap L) \leq |H| < \infty$. This already shows that $\lim(\fa)\subset \pi^{-1}(\sigma)$.

\end{proof}

In the above theorem the ramification set is irreducible but we make no hypothesis on the type of singularities it may have. The price we pay is to assume that there is a non-vertical invariant algebraic  hypersurface. This condition is natural in the following situation.
Let $R$ be a rational function $R \colon \mathbb P^m \times \mathbb P^n \to \mathbb P^n$. We shall define the {\it ramification set} of $R$ (with respect to the vertical fibration) as the set $\sigma \subset \mathbb P^m$ of points $p$ for which the fiber $\{p \} \times \mathbb P^n$ is not transverse to $\fa$ (this means that there is some point $q \in \mathbb P^n$ for which the leaf of $\fa$ through  $(p,q)$ is not transverse to the fiber $\{p\}\times \mathbb P^n$).
In general $\sigma$ is an algebraic subset of codimension $\geq 1$  in the projective plane $\mathbb P^m$.

Let us consider an analytic family of  Riccati foliations in $\mathbb P^1 \times \mathbb P^1$ given in affine coordinates by the 1-forms $\omega_t=(1 + tp(x))dy - t(a(x)y^2 + b(x)y)dx$ where
$p(x), a(x), b(x)$ are polynomials. If $\fa_t$ denotes the foliation on $\mathbb P^1 \times \mathbb P^1$ induced by $\omega_t$ then $\fa_0: \omega_0=dy$ is the horizontal fibration, given by the second coordinate projection. The ramification set of $\fa_0$ is irreducible
empty, while for $t\ne 0$ the ramification set of $\fa_t$ is given by $p(x)=-1/t$ and possibly the point at the infinity $x=\infty$.  This set is not irreducible for many choices of the coefficients $a, b, p$. In general $\fa_t$ is not analytically equivalent to a trivial foliation in $\mathbb P^1 \times \mathbb P^1$. Thus, an irreducible ramification set can deform into a reducible ramification set during a deformation by Riccati foliations.

 Taking this into account we can state:
\begin{Theorem}
\label{Theorem:Riccati2}{\rm
Let $\fa$ be  the foliation by level surfaces of a rational function $R\colon \mathbb P^2\times \mathbb P^n \to \mathbb P^n$. Assume that codimension one component of the ramification set $\sigma\subset \mathbb P^2$ of $R$ is empty or  irreducible (not necessarily smooth nor normal crossing type) of  degree   $p^{s}$ for some prime number $p$ and some $s\in \mathbb N$. Let now $\{\fa_t\}_{t \in \mathbb D}$ be an  analytic deformation of $\fa=\fa_0$ by
  Riccati foliations on $\mathbb P^2 \times \mathbb P^n$ with irreducible ramification set $\sigma(t)\subset \mathbb P^2$.  Assume that there is some level $(R=c)$ of $R$ which is invariant by each foliation $\fa_t$.
  Then the global holonomy of  $\fa_t$ is finite cyclic for each $t$ close to $0$. In particular, the leaves of $\fa_t$ are closed in  $(\mathbb P^2 \setminus \sigma(t)) \times \mathbb P^n$, i.e., $\lim(\fa_t)\subset \sigma(t) \times \mathbb P^n,$ for all $t$ close to $0$. If $R$ is the second projection $\mathbb P^2 \times \mathbb P^n \to \mathbb P^n$ then $\fa_t$ is analytically conjugate to $\fa$ in $(\mathbb P^2 \setminus \sigma(t)) \times \mathbb P^n$.
}
\end{Theorem}

\begin{proof}
First we consider the case where $\sigma(\fa)$ has codimension $\geq 2$. We denote by
$P_1\colon \mathbb P^m \times \mathbb P^n \to \mathbb P^m$ the first coordinate projection. Given a leaf
$L\in\fa$ the restriction $P_1\big|_L\colon L \to \mathbb P^m \setminus \sigma(\fa)$ is a covering map. The fundamental group $\pi_1(\mathbb P^m \setminus \sigma(\fa))$ is trivial because $\codim\; \sigma(\fa)\geq 2$ in $\mathbb P^m$. This implies that $P_1\big|_L$ is a holomorphic diffeomorphism from $L$ to $B_0:=\mathbb P^m \setminus \sigma(\fa)$.
By Hartogs' extension theorem, applied to the inverse $(P_1\big|_L)^{-1}$, again using the fact that $\codim \;\sigma(\fa) \geq 2$, we can extend $P_1\big|L$ to a holomorphic diffeomorphism between $\overline{L}$ and $\mathbb P^m$. Moreover, by this extension we conclude that indeed $\sigma(\fa)=\emptyset$. Thus the function $R$ has levels that correspond to the horizontal fibration, i.e., it depends only on the second coordinate.
If we take $R$ as a primitive rational function then we may assume that $R(x,y)=y$ in coordinates $(x,y)\in \mathbb P^m \times \mathbb P^n$.
Now we assume that $\sigma(\fa)\ne \emptyset$ is irreducible of degree $p^s$.
Given $t$ close enough to $0$, by hypothesis the ramification set $\sigma(t)$ of the Riccati foliation $\fa_t$ is still irreducible and therefore has degree $p^s$. Indeed, $\{\sigma(t)\}_{t\in D}$ defines an analytic family of irreducible algebraic curves in $\mathbb P^2$. In particular, the fundamental groups $\pi_1(\mathbb P^2 \setminus \sigma(t))$ are the same. This implies that the  holonomy group $\Hol(\fa_t,L_c)$ of the common leaf contained in $L_c\subset R_c$ is an analytic deformation of the holonomy group
of $\fa_0=\fa$. Let us be more precise.
Given a non-invariant fiber $F_{x_0}: \{x_0\}\times\mathbb P^n$  and the common invariant level $R_c: (R=c)$  there is a finite intersection set $R_c\cap F_{x_0}=\{y_1,\ldots,y_r\}$. Denote by $H_t$ the global holonomy group of $\fa_t$ given by the representation $H_t\subset \Aut(F_{x_0})$. Given the point $y_1 \in F_{x_0}$ we consider the holonomy group $H_{t,y_1}:=\Hol(\fa_t,L_{t,y_1})$ of the leaf $L_{t,y_1}$ of $\fa_t$ passing through $y_1$ and calculated with respect to the transverse section contained in the fiber $F_{x_0}$. By hypothesis $R_c$ is invariant by $\fa_t$ so that
 $L_{t,y_1}\subset R_c$. By Theorem~\ref{Theorem:stability} each group $H_{t,y_1}$ is finite cyclic of uniformly bounded order  for $t$ close to $0$.
Given $t\approx 0$ and a global holonomy map $f \in H_t$ we have that $f(R_c\cap F_{x_0})=R_c\cap F_{x_0}=\{y_1,\ldots,y_r\}$. Thus $f^{r!}(y_1)=y_1$. This implies that $f^{r!}\in H_{t,y_1}$ and since this group is cyclic finite of uniformly bounded order, this shows that
each map $f\in H_t$ has a uniformly bounded finite order. Indeed, since each holonomy map
in $H_{t,y_1}$ comes from a global holonomy map in $H_t$ this shows that each global holonomy group $H_t$ is finite cyclic of a uniformly bounded order.
The limit set part is proved as before.
Assume now that $R$ is the second coordinate projection (for instance if $\sigma(\fa)$ has codimension $\geq 2$). Then the global holonomy group $H_0$ is trivial. By Theorem~\ref{Theorem:stability} the holonomy groups $H_{t,y_1}$ are trivial. Also note that $r=1$ i.e, $R_c\cap F_{x_0}$ consists of a single point. Similarly to above we then conclude that $H_t$ is trivial for each $t\approx 0$. This shows that the foliation $\fa_t$ is equivalent to $\fa_0$ in $\mathbb P^2 \setminus \sigma(t) \times \mathbb P^n$.
\end{proof}

\subsection{Case $m \geq 2, n \geq 2$}

Now we study Riccati foliations in $\mathbb P^m \times \mathbb P^n, m \geq, \, n \geq 2$ under some hypothesis on the ramification set $\sigma\subset \mathbb P^m$.
\begin{proof}[Proof of Theorem~\ref{Theorem:Riccati3}]
The proof is pretty much the same given for the other cases $m=2, n=1$. For using the irreducibility of $\sigma_1 \subset \mathbb P^m$ it is enough to apply Lefshetz hyperplane section theorem together with Deligne's theorem for $m=2$. All the rest goes as in Theorem~\ref{Theorem:Riccati2}.
\end{proof}

\section{Some examples and comments}

Let us address  the questions  mentioned in the Introduction.
We summarize our conclusions as follows:

\begin{Proposition}
\label{Proposition:examples}{\rm Regarding irreducible groups of germs we have:
\begin{enumerate}
\item Conditions (a) and (b) in Definition~\ref{Definition:irreduciblegroup} are not equivalent.
\item A finite abelian subgroup $G\subset \Diff(\mathbb C^n,0)$ is not necessarily irreducible.

\item An irreducible subgroup $G\subset \Diff(\mathbb C^n,0)$ is not necessarily finite.

\item A \underline{finite} irreducible subgroup $G\subset \Diff(\mathbb C^n,0)$ is not necessarily cyclic.

\end{enumerate}
}
\end{Proposition}

\begin{proof}
We start with (1). We look at the linear case. Let

\[
A
=\left(
  \begin{array}{cc}
    2 & 0\\
     0 &  1
  \end{array}
\right),\hspace{0.5 cm}
B = \left(
  \begin{array}{cc}
    2 & 0\\
     1 &  1
  \end{array}
\right)
\]

and $G =<A,B>\subset \GL(2,\mathbb C)$. Let us now see that $A$ and $B$ are conjugate in $G$, since

$$H=B^{-1}A=\left( \begin{array}{cc}  \frac{1}{2} & 0\\
     -\frac{1}{2} &  1
  \end{array}
\right)  \left(
  \begin{array}{cc}
    2 & 0\\
     0 &  1
  \end{array}
\right)=\left(
  \begin{array}{cc}
    1 & 0\\
     -1 &  1
  \end{array}
\right)\in G$$ and
$$AH= \left(
  \begin{array}{cc}
    2 & 0\\
     0 &  1
  \end{array}
\right)\left(
  \begin{array}{cc}
    1 & 0\\
     -1 &  1
  \end{array}
\right)= \left(
  \begin{array}{cc}
    2 & 0\\
     -1 &  1
  \end{array}
\right)= \left(
  \begin{array}{cc}
    1 & 0\\
     -1 &  1
  \end{array}
\right) \left(
  \begin{array}{cc}
    2 & 0\\
     1 &  1
  \end{array}
\right)=HB.$$ Therefore {\em $G$ satisfies  }(b)\;{\em  but not }(a).

Now we consider

 \[
\tilde A
=\left(
  \begin{array}{ccc}
    1 & 0&0\\
     0 &  1&0\\1&0&1
  \end{array}
\right),\hspace{0.5 cm}\tilde B = \left(
  \begin{array}{ccc}
    1 & 0&0\\
     0 &  1&0\\0&1&1
  \end{array}
\right),\hspace{0.5 cm}\tilde C
=\left(
  \begin{array}{ccc}
    1 & 0&0\\
     0 &  1&0\\-1&-1&1
  \end{array}
\right)
\]

and $\tilde G =<\tilde A,\tilde B,\tilde C>\subset \GL(3,\mathbb C)$. Note that $\tilde G$ satisfies (a). Let us see that it does not satisfy (b). Indeed, since the generators of $\tilde G$ are upper triangular matrices then any element of $\tilde G$ must be an upper triangular matrix. Then it is not difficult to see that there is no conjugation in $\tilde G$ between  $\tilde A$ and $\tilde B$.

Now we show (2). Put
\[
A
=\left(
  \begin{array}{cc}
    i & 0\\
     0 &  -i
  \end{array}
\right),\hspace{0.5 cm}
B = \left(
  \begin{array}{cc}
    0 & 1\\
     1 &  0
  \end{array}
\right)
\]

and consider $ G = <A,B>\subset \GL(2,\mathbb C)$. Note that $$A^4 = B^2 = (AB)^2 = (BA)^2 = (BA^2)^2 =\Id.$$
Also $AB^2 = A$, $ABA = B$, $BAB = A^3$, $BA^2B = A^2$ and $BA^3B = A$. Thus $ G$ is finite:

\begin{multline*}
   G
=\left\lbrace
\left(
  \begin{array}{cc}
    1 & 0\\
     0 &  1
  \end{array}
\right),\,
\left(
  \begin{array}{cc}
    i & 0\\
     0 &  -i
  \end{array}
\right),\,
\left(
  \begin{array}{cc}
    0 & 1\\
     1 &  0
  \end{array}
\right),\,
\left(
  \begin{array}{cc}
    -1 & 0\\
     0 &  -1
  \end{array}
\right),\,
\left(
  \begin{array}{cc}
    -i & 0\\
     0 &  i
  \end{array}
\right),\,
\left(
  \begin{array}{cc}
    0 & -1\\
     -1 &  0
  \end{array}
\right), \right.\\ \left.
  \,
\left(
  \begin{array}{cc}
    0 & i\\
     -i &  0
  \end{array}
\right),\,
\left(
  \begin{array}{cc}
    0 & -i\\
     i &  0
  \end{array}
\right) \right\rbrace.\hspace{5 cm}
\end{multline*}

But $ G$ is not irreducible because $A$ and $B$ are not conjugate since $A$ and $B$
have different orders ($A^4 = \Id$ and $B^2 = \Id$).\\
The group  $ G$ above is not abelian. We may ask then what happens in the abelian case.
Again the  answer is negative: let $\lambda\in \mathbb C$ be such that $\lambda^n= 1$ for some $n\in\mathbb{N}$ $n>4$ and let

\[
A
=\left(
  \begin{array}{cc}
    i & 0\\
     0 &  -i
  \end{array}
\right)\hspace{0.5 cm}
\mbox{and}\hspace{0.5 cm}
B = \left(
  \begin{array}{cc}
    1 & 0\\
     0 &  \lambda
  \end{array}
\right).
\]

We take $ G = <A,B>\subset \GL(2,\mathbb C)$ then $ G$ is abelian and finite but not irreducible ($A^4 = \Id = B^n$, $n\ne4$).
Note that $ G$ is finite abelian,  not  generated by a single element, and it is not irreducible.\\

Let us now show (3).
Put

\[
A
=\left(
  \begin{array}{cc}
    1 & 0\\
     0 &  -1
  \end{array}
\right),\hspace{0.5 cm}
B = \left(
  \begin{array}{cc}
    -\frac{1}{2} & 1\\
     \frac{3}{4} &  \frac{1}{2}
  \end{array}
\right)
\hspace{0.5 cm}
\mbox{and}\hspace{0.5 cm}
C = \left(
  \begin{array}{cc}
    -\frac{1}{2} & 2\\
     \frac{3}{8} &  \frac{1}{2}
  \end{array}
\right).
\]

Note that $A^2 = B^2 = C^2 = \Id$, now take

\[
H = B^{-1}A = \left(
  \begin{array}{cc}
    -\frac{1}{2} & -1\\
     \frac{3}{4} &  -\frac{1}{2}
  \end{array}
\right)
\hspace{0.5 cm}
\mbox{and}\hspace{0.5 cm}
T = C^{-1}A = \left(
  \begin{array}{cc}
    -\frac{1}{2} & -2\\
     \frac{3}{8} &  -\frac{1}{2}
  \end{array}
\right).
\]

Then $H,T\in G$, $HBH^{-1} = A$, $TCT^{-1} = A$ and $T^{-1}HB(T^{-1}H)^{-1} = C$. Thus, the group:

\[
G = <A,B,B,C,C,A>
\]

is irreducible and not cyclic. Now we will verify that the group $G$ is not finite. Indeed, we take

\[
BC = \left(
  \begin{array}{cc}
    -\frac{1}{2} & 1\\
     \frac{3}{4} &  \frac{1}{2}
  \end{array}
\right)
\cdot\left(
  \begin{array}{cc}
    -\frac{1}{2} & 2\\
     \frac{3}{8} &  \frac{1}{2}
  \end{array}
\right) =
\left(
  \begin{array}{cc}
    \frac{5}{8} & -\frac{1}{2}\\
     -\frac{3}{16} &  \frac{7}{4}
  \end{array}
\right).
\]

Now we will study the signal of entries in the matrix $BC$ without importing its values. So we can represent the matrix $(BC)^2$:
\[
(BC)^2 =
\left(
  \begin{array}{cc}
    + & -\\
     - &  +
  \end{array}
\right)\cdot
\left(
  \begin{array}{cc}
    + & -\\
     - &  +
  \end{array}
\right) =
\left(
  \begin{array}{cc}
    (+\cdot +) + (-\cdot -) & (+\cdot-) + (-\cdot +)\\
     (-\cdot +) + (+\cdot -) &  (-\cdot -) + (+\cdot +)
  \end{array}
\right) =
\left(
  \begin{array}{cc}
    + & -\\
     - &  +
  \end{array}
\right).
\]

Thus each entries in the matrix increases in module. Therefore, there is no $n\in\mathbb{N}$ such that $(BC)^n = \Id$.\\

Finally, we address (4).
Let
\[
A
=\left(
  \begin{array}{cc}
    1 & 0\\
     0 &  -1
  \end{array}
\right),\hspace{0.5 cm}
B = \left(
  \begin{array}{cc}
    -\frac{1}{2} & 1\\
     \frac{3}{4} &  \frac{1}{2}
  \end{array}
\right)
\hspace{0.5 cm}
\mbox{and}\hspace{0.5 cm}
H = B^{-1}A = \left(
  \begin{array}{cc}
    -\frac{1}{2} & -1\\
     \frac{3}{4} &  -\frac{1}{2}
  \end{array}
\right).
\]

Note that $A^2 = B^2 = \Id$, $HBH^{-1} = A$ and $H^{-1}AH = B$. Now consider the $4\times 4$ matrices
formed by $2\times2$ diagonal blocks. We denote by $D_{A,A}$ the matrix with diagonal blocks $A$.
Now consider
\[
G = <D_{A,A}, D_{A,B}, D_{B,B}, D_{B,A}>\subset \GL(4,\mathbb C).
\]

We have that $D_{A,A}\cdot D_{A,B}\cdot D_{B,B}\cdot D_{B,A} = D_{A^2B^2,AB^2A} = \Id$.
The generators are conjugates 2 to 2 in the groups, without loss of generality we verify that $D_{A,A}$ is conjugates to $D_{A,B}$
in the groups. Indeed, take
\[
T= D_{A,B}^{-1}\cdot D_{A,A} = D_{A^{-1}A,B^{-1}A} = D_{\Id,H}\in G.
\]
Then
\[
T\cdot D_{A,B}\cdot T^{-1} = D_{\Id,H}\cdot D_{A,B}\cdot D_{\Id,H}^{-1} = D_{\Id\cdot A\cdot Id, HBH^{-1}} = D_{A,A}.
\]
This ends the proof.
\end{proof}
\bibliographystyle{amsalpha}

\end{document}